\documentclass[12pt, reqno, a4paper]{amsart}

\usepackage{ amssymb, amsmath, enumerate, amsfonts, amsthm, mathrsfs, url, bm, mathtools}

\usepackage{xcolor}  	
\usepackage{hyperref}
\hypersetup{
colorlinks,
   linkcolor={cyan!80!black},
   citecolor={cyan!80!black},
 urlcolor={cyan!80!black}
}

\usepackage{color}

\usepackage[margin=1in]{geometry}

\RequirePackage{doi}

\usepackage{amscd}
\usepackage{amsfonts}
\usepackage{float}
\usepackage{color}
\usepackage[
backend=biber,
style=alphabetic,
]{biblatex}
\usepackage{bookmark}

\renewbibmacro{in:}{}
\DeclareFieldFormat{title}{#1}

\DeclareFieldFormat[article]{title}{\mkbibemph{#1}}       
\DeclareFieldFormat[incollection]{title}{\mkbibemph{#1}}  
\DeclareFieldFormat[book]{title}{\mkbibemph{#1}}          
\DeclareFieldFormat[incollection]{booktitle}{#1}          
\DeclareFieldFormat[article]{journaltitle}{#1}            

\AtEveryBibitem{%
  \ifentrytype{misc}{\DeclareFieldFormat{title}{\mkbibemph{#1}}}{}}

\DeclareFieldFormat{eprint:eprint}{arXiv:\href{https://arxiv.org/abs/#1}{#1}}

\DeclareFieldFormat[inproceedings]{title}{\mkbibemph{#1}}

\DeclareFieldFormat[inproceedings]{booktitle}{#1}

\usepackage{amssymb}

\newtheorem{theorem}{Theorem}[section]
\newtheorem{lemma}{Lemma}[section]

\newtheorem{proposition}{Proposition}[section]

\theoremstyle{definition}

\theoremstyle{remark}
\newtheorem{remark}{Remark}[section]

\numberwithin{equation}{section}

\renewcommand{\leq}{\leqslant}
\renewcommand{\geq}{\geqslant}

\newcommand{\Li}{\operatorname{Li}}

 \newcommand{\EWp}{\mathbb E^{\rm WP}_{M \in \mathcal{M}_g}} 

\begin{document}

\title[A central limit theorem for prime geodesics]{A central limit theorem for prime geodesics on random surfaces of large genus}


\author{Sun-Kai Leung}
\address{Mathematical Institute, University of Oxford, Andrew Wiles Building\\ Radcliffe Observatory Quarter, Woodstock Rd\\
Oxford OX2 6GG\\
United Kingdom}
\email{sunkaileung@gmail.com}



\date{}

\dedicatory{}


\begin{abstract}
We show that, for Weil--Petersson random closed hyperbolic surfaces of large genus, the normalized weighted count of prime geodesics with norm in the interval $(X,X+H]$ is asymptotically Gaussian, provided $H \leq X$ and $H/\log X\to\infty$ as $X\to\infty$. In particular, it applies to intervals which are not necessarily short.
\end{abstract}

\maketitle

\section{Introduction}

The distribution of primes has fascinated mathematicians for centuries. In 1896, Hadamard and de la Vallée Poussin independently proved the \textit{Prime Number Theorem}. Let $\pi(x):=\#\{p\leq x\}$ denote the prime counting function. The theorem states that 
\begin{align*}
\pi(x) \sim \Li(x):=\int_2^x \frac{dt}{\log t}  \qquad \text{as } x\to\infty,
\end{align*}
or equivalently, the weighted count of prime powers satisfies
\begin{align*}
\psi(x):=\sum_{n\leq x}\Lambda(n)\sim x  \qquad \text{as } x\to\infty,
\end{align*}
where $\Lambda$ is the von Mangoldt function, defined by
\begin{align*}
\Lambda(n):=
\begin{cases}
\log p & \mbox{\rm if $n=p^k$ for some $k\in\mathbb{N}$,} \\
\hfil 0 & \mbox{\rm otherwise.}
\end{cases}
\end{align*}

The probabilistic interpretation of the distribution of primes, particularly  in short intervals, was initiated by Gauss in his famous letter to Encke, and much later developed by Cramér \cite{Cramér1936} (see also \cite{MR1349149}). Given $1 \leq h \leq x \leq X,$ let $\pi(x;h):=\pi(x+h)-\pi(x)$ and $\psi(x;h):=\psi(x+h)-\psi(x),$ which can be viewed as random variables when $x$ is chosen uniformly at random from $[X,2X].$ By the prime number theorem, the mean of $\psi(x;h)$ satisfies
\begin{align*}
\frac{1}{X}\int_X^{2X}\psi(x;h)\,dx \sim h  \qquad \text{as } X\to\infty,
\end{align*}
provided that $h=o(X).$ Assuming the Riemann hypothesis and the strong pair correlation conjecture \cite{MR0337821}, Goldston and Montgomery \cite{MR1018376} showed that the ``variance" of $\psi(x;h)$ satisfies 
\begin{align*}
\frac{1}{X}\int_X^{2X}(\psi(x;h)-h)^2\,dx \sim h\log \frac{X}{h} \qquad \text{as } X\to\infty,
\end{align*}
provided that $X^{\epsilon} \leq h \leq X^{1-\epsilon}$ for any fixed $\epsilon>0.$ 

Assuming a uniform variant of the Hardy--Littlewood prime $k$-tuple conjecture \cite{MR1555183}, Gallagher \cite{MR409385} showed convergence in distribution to a Poisson random variable with parameter $\lambda$ 
\begin{align*}
\pi(x;h) \xrightarrow{\;\;d\;\;} \operatorname{Poisson}(\lambda) \qquad \text{as } X\to\infty,
\end{align*} 
provided that $h=\lambda \log x$ for some fixed $\lambda>0$. Recently, the author \cite{MR4865836} extended the result to a joint Poisson limit theorem for primes in disjoint intervals of length $\asymp \log X$.

Assuming another uniform variant of the Hardy--Littlewood prime $k$-tuple conjecture with power-saving error terms, Montgomery and Soundararajan \cite{MR2104891} showed convergence in distribution to a standard Gaussian 
\begin{align} \label{eq:montsound}
\frac{\psi(x;h)-h}{\sqrt{h\log (X/h)}}
\xrightarrow{\;\;d\;\;} \mathcal{N}(0,1) \qquad \text{as } X\to\infty,
\end{align}
provided that $h/\log X\to\infty$ but $\log h=o(\log X)$. See also \cite{MR5029894} in the logarithmic scale for the range $h=\delta x$, where $\delta>0$ is sufficiently small.

In this paper, we shall establish a central limit theorem for closed geodesics on random surfaces of large genus, analogous to that of Montgomery and Soundararajan for primes.

Let \(M\) be a closed hyperbolic surface of genus \(g\geq 2\). We define
the \emph{norm} of a closed geodesic \(\gamma\) on \(M\) by $N(\gamma):=\exp(\ell(\gamma)),$ where \(\ell(\gamma)\) denotes the length of \(\gamma\).
Let 
\begin{align*}
\Pi_M(X):=\# \{ \gamma \subseteq M \,:\, \gamma \text{ primitive, closed, oriented, and } N(\gamma) \leq X \}
\end{align*}
denote the prime geodesic counting function. Also, we define the weighted count of closed, oriented geodesics $\gamma$ on $M$ with norm $N(\gamma)\leq X$ by
\begin{align*}
\Psi_M(X):=\sum_{\substack{\gamma \subseteq M\\N(\gamma)\leq X}}\Lambda(\gamma),
\end{align*}
where the \textit{von Mangoldt weight} $\Lambda$ is defined by
$
\Lambda(\gamma)
:=
\log N(\gamma_0)
=
\ell(\gamma_0)
$
whenever \(\gamma=\gamma_0^k\), with \(\gamma_0\) primitive and
\(k \in \mathbb{N}\). The \textit{Prime Geodesic Theorem} \cite{MR109212} states that 
\begin{align*}
\Pi_M(X) \sim \operatorname{Li}(X) \qquad \text{as } X\to\infty,
\end{align*}
or equivalently, the weighted count of closed geodesics satisfies
\begin{align*}
\Psi_M(X) \sim X \qquad \text{as } X\to\infty.
\end{align*}

Mirzakhani \cite{MR2264808,MR2415399} pioneered the study of the statistics of closed geodesics on random surfaces sampled from the \textit{moduli space} $\mathcal{M}_g$ of large genus $g$ with respect to the \textit{Weil--Petersson probability measure} (see \cite{MR4108090} and \cite{MR4670581} for exposition). 

Following \cite{rudnick2026closedgeodesicsshortintervals}, to count geodesics with norm in an interval, we denote
\begin{align*}
\Pi_M(X;H):=\Pi_M(X+H)-\Pi_M(X)
\end{align*}
and
\[
\Psi_M(X;H)
:=
\Psi_M(X+H)-\Psi_M(X)
\]
for \(1\leq H\leq X\), which can be viewed as random variables when $M$ is sampled from the moduli space $\mathcal{M}_g$ of large genus $g$ with respect to the Weil--Petersson probability measure. Rudnick \cite{rudnick2026closedgeodesicsshortintervals} showed that the mean of $\Psi_M(X;H)$ satisfies
\begin{align*}
 \EWp \Psi_M(X;H) \sim H \qquad \text{as $g\to\infty$ followed by $X \to \infty$}.\footnotemark
\end{align*}
\footnotetext{In this paper, we write $\EWp$ instead of $\mathbb{E}^{\rm WP}_g$ as the expectation with respect to the Weil--Petersson probability measure to emphasize that the randomness is in $M$, rather than in $X$ or $H$.} He also showed that the ``variance" of $\Psi_M(X;H)$ satisfies
\begin{align*}
\EWp (\Psi_M(X;H)-H)^2 \sim 2H\log X \qquad \text{as $g\to\infty$ followed by $X \to \infty$}.
\end{align*}

\begin{remark} \label{rmk}
A careful inspection of the proofs of \cite[Proposition 3.1 \& Theorem 4.1]{rudnick2026closedgeodesicsshortintervals} shows that the weaker assumption \(H\leq X\) suffices in place of \(H=o(X)\).
\end{remark}

Remarkably, Mirzakhani and Petri \cite{MR4046008} established a joint Poisson limit theorem for closed, unoriented geodesics on random surfaces of large genus with norm in disjoint intervals of length $\asymp \log X$. In particular, by taking orientation into account, we have convergence in distribution to twice a Poisson random variable with parameter $\lambda_{\kappa}$ 
\begin{align*}
\Pi_M(X;H) \xrightarrow{\;\;d\;\;} 2 \cdot \operatorname{Poisson}(\lambda_{\kappa}) \qquad \text{as } g\to\infty,
\end{align*} 
provided that $H=\kappa \log X$ for some fixed $\kappa>0,$ where
\begin{align*}
\lambda_{\kappa}=\lambda_{\kappa}(X)&=\frac{1}{2} \int_{\log X}^{\log (X+H)} \left(e^{\frac{x}{2}}-e^{-\frac{x}{2}}\right)^2\frac{dx}{x} \\
&=\left( 1+O\left(\frac{1}{\log X} \right) \right)\frac{\kappa}{2}.
\end{align*}
Letting $X \to \infty,$ we obtain
\begin{align*}
\Pi_M(X;H) \xrightarrow{\;\;d\;\;} 2 \cdot \operatorname{Poisson}(\kappa/2) 
\qquad \text{as $g\to\infty$ followed by $X \to \infty$},
\end{align*} 
provided that $H=\kappa \log X$ for some fixed $\kappa>0.$

As the parameter increases, the Poisson random variable evolves into a Gaussian. More precisely, we have convergence in distribution to a standard Gaussian
\begin{align*}
\frac{\operatorname{Poisson}(\lambda)-\lambda}{\sqrt{\lambda}} 
\xrightarrow{\;\;d\;\;} \mathcal{N}(0,1) \qquad \text{as $ \lambda\to\infty.$}
\end{align*}
It is therefore natural to expect 
\begin{align*}
\frac{\Pi_M(X;H)-(H/\log X)}{\sqrt{2H/\log X}} 
\xrightarrow{\;\;d\;\;} \mathcal{N}(0,1)
\qquad \text{as $g\to\infty$ followed by $X \to \infty$},
\end{align*}
provided that $H/\log X\to\infty.$ Correspondingly, for the weighted counting function, one expects
\begin{align*}
\frac{ \Psi_{M}(X;H)-H}{\sqrt{2H\log X}}
\xrightarrow{\;\;d\;\;} \mathcal{N}(0,1) \qquad \text{as $g\to\infty$ followed by $X \to \infty$}
\end{align*}
under the same condition.\footnote{Rudnick \cite{rudnick2026closedgeodesicsshortintervals} gave a heuristic explanation of the additional factor $2$ in the variance in terms of the expected GOE statistics for the Laplace spectrum of generic closed hyperbolic surfaces, in contrast to the expected GUE statistics for the nontrivial zeros of the Riemann zeta function, as suggested by Montgomery's pair correlation conjecture. Here, we attribute this additional factor of $2$ to the two possible orientations of a closed geodesic, whereas primes have no such distinction.} In this paper, we confirm this prediction as our main result.

\begin{theorem} \label{thm:main}
Let $1 \leq H \leq X$ satisfy 
$H/\log X \to \infty$ as $X \to \infty.$ Suppose for each $g \geq 2,$ the random surface $M$ is sampled from $\mathcal{M}_g$ with respect to the Weil--Petersson probability measure. Then as $g \to \infty$ followed by $X \to \infty,$ we have convergence in distribution to a standard Gaussian
\begin{align*}
\frac{ \Psi_{M}(X;H)-H}{\sqrt{2H\log X}}
\xrightarrow{\;\;d\;\;} \mathcal{N}(0,1),
\end{align*}
i.e., for any bounded continuous function $F:\mathbb{R}\to\mathbb{C}$, we have
\begin{align*}
\lim_{X\to\infty}\lim_{g\to\infty}
\EWp
F\left(\frac{\Psi_M(X;H)-H}{\sqrt{2H\log X}}\right)
=
\frac{1}{\sqrt{2\pi}}\int_{\mathbb{R}}F(t)e^{-\frac{1}{2}t^2}dt.
\end{align*}
\end{theorem}

\begin{remark}
Rudnick and Wigman \cite{MR4682953} proved a central limit theorem for the Laplace spectrum of random surfaces of large genus, rather than for the length spectrum considered here. In fact, we shall adapt their argument to prove Theorem \ref{thm:main}.
\end{remark}

\begin{remark}
Adapting the argument for Theorem \ref{thm:main} presented in this paper with  \cite[Theorem 4.1]{MR4046008} replaced by \cite[Corollary 2.1]{MR4865836}, one can show that assuming a uniform variant of the prime $k$-tuple conjecture, if $x$ is chosen uniformly at random from $[X,2X]$ for some large $X$, then
\begin{align*}
\frac{\psi(x;h)-h}{\sqrt{h\log (X/h)}}
\xrightarrow{\;\;d\;\;} \mathcal{N}(0,1)
\qquad \text{as $X\to\infty$ followed by $\lambda\to\infty$},
\end{align*}
where $h=\lambda\log x$. This recovers (\ref{eq:montsound}) in a narrow range.
\end{remark}



\noindent\textit{Notation.} 
Throughout the paper, we use the standard big $O,$ little $o$ notations, the asymptotic notation $\sim,$ the Vinogradov notation $\ll,$ and the Hardy notation $\asymp,$ where the implied constants depend only on the subscripted parameters. We denote by $\mathbb{P}_{\ast \in \mathcal{M}_g}^{\rm WP}$ the Weil--Petersson probability measure on the moduli space $\mathcal{M}_g,$ and by $\mathbb{E}_{\ast \in \mathcal{M}_g}^{\rm WP}$ the corresponding expectation.

\section{Poisson Approximation for the Length Spectrum}

In this section, we approximate \(\Psi_M(X;H)\) by a functional of the
\textit{Mirzakhani--Petri point process}.

Following \cite{rudnick2026closedgeodesicsshortintervals}, given $1 \leq H \leq X,$ let $A:=\log X, B:=\log(X+H),$ and
\begin{align*}
f_{X;H}(\ell):=
\begin{cases}
2\ell \left( \left\lfloor \frac{B}{\ell} \right\rfloor
-  \left\lfloor \frac{A}{\ell} \right\rfloor
\right) & \mbox{\rm if $\ell>0$} \\
\hfil 0 & \mbox{\rm if $\ell=0.$}
\end{cases}
\end{align*}
Also, let \(\mathcal{N}\) denote the space of locally finite point
configurations on \([0,\infty)\), identified with their associated
counting measures and endowed with the vague topology. We define
the functional \(\mathcal{F}_{X;H}:\mathcal{N}\to\mathbb{R}\) by
\[
\mathcal{F}_{X;H}(\mathcal{L})
:=
\sum_{\ell\in\mathcal{L}} f_{X;H}(\ell) \qquad \text{for } \mathcal{L}\in\mathcal{N}.
\]

Given $g \geq 2,$ define the length-spectrum map $\mathcal{L}_g: \mathcal{M}_g \rightarrow \mathcal{N}$ by
\begin{align*}
\mathcal{L}_g(M):=\{ \ell(\gamma) \,:\, \gamma \subseteq M  \text{ primitive, unoriented}\} \qquad \text{for } M \in \mathcal{M}_g,
\end{align*}
where the lengths are counted with multiplicity. Then by definition, we have
\begin{align*}
\Psi_M(X)&= 2 \sum_{\substack{\gamma \subseteq M \\\text{primitive} \\ \text{unoriented}}} \sum_{\substack{k \geq 1\\ k \ell(\gamma) \leq \log X}} \ell(\gamma) \\
&=2 \sum_{\substack{\gamma \subseteq M \\\text{primitive} \\ \text{unoriented}}}  
\ell(\gamma) \left\lfloor \frac{\log X}{\ell(\gamma)} \right\rfloor \\
&= \sum_{\ell \in \mathcal{L}_g(M)} 2\ell 
\left\lfloor \frac{A}{\ell} \right\rfloor.
\end{align*}
Similarly, we have
\begin{align*}
\Psi_M(X+H)=  \sum_{\ell \in \mathcal{L}_g(M)} 2\ell 
\left\lfloor \frac{B}{\ell} \right\rfloor,
\end{align*}
and therefore
\begin{align} \label{eq:psi=F}
\Psi_M(X;H)=\mathcal{F}_{X;H}(\mathcal{L}_g(M)).
\end{align}

To approximate \(\Psi_M(X;H)\), let ${\rm Poisson} (\nu_{\rm MP})$ denote the Poisson point process on $[0,\infty)$ with intensity
\begin{align*}
\nu_{\rm MP}(dx):=\frac{1}{2}\left(e^{\frac{x}{2}}-e^{-\frac{x}{2}}\right)^2\frac{dx}{x}.
\end{align*}
For convenience, we write $\mathcal{L}_{\rm MP}:={\rm Poisson} (\nu_{\rm MP}),$ and refer to it as the Mirzakhani--Petri point process.

\begin{proposition} \label{prop:compare}
Given $1 \leq H \leq X,$ let $\Sigma_{\rm MP}(X;H):=\mathcal{F}_{X;H}(\mathcal{L}_{\rm MP}).$ Suppose for each $g \geq 2,$ the random surface $M$ is sampled from $\mathcal{M}_g$ with respect to the Weil--Petersson probability measure. Then as $g \to \infty,$ we have convergence in distribution
\begin{align*}
\Psi_{M}(X;H) \xrightarrow{\;\;d\;\;}  \Sigma_{\rm MP}(X;H),
\end{align*}
i.e., for any bounded continuous function $F:\mathbb{R} \to \mathbb{C},$ we have
\begin{align*}
\lim_{g \to \infty} \EWp F(\Psi_{M}(X;H)) = 
\mathbb{E} F(\Sigma_{\rm MP}(X;H)).
\end{align*}

\begin{remark}
In fact, the proposition extends from \(f_{X;H}\) to any bounded Borel
function with bounded support whose set of discontinuities has
\(\nu_{\mathrm{MP}}\)-measure zero.
\end{remark}

\begin{proof}
In the language of point processes, Mirzakhani and Petri \cite[Theorem 4.1]{MR4046008} established the convergence in distribution
\begin{align*}
\mathcal{L}_g \xrightarrow{\;\;d\;\;} \mathcal{L}_{\rm MP}
\end{align*}
as $g \to \infty$ (see also \cite[Section 3.2]{MR4682953}). Note that the function $f_{X;H}$ is bounded and compactly supported. Also, its set of discontinuities is contained in 
\begin{align*}
D_{X;H}:=\{0\} \cup \left\{  \frac{A}{n} \,:\, n \geq 1 \right\} \cup \left\{  \frac{B}{n} \,:\, n \geq 1 \right\},
\end{align*}
which is countable. Since $\nu_{\rm MP}$ is absolutely continuous, we have
$\nu_{\rm MP}(D_{X;H})=0,$ and therefore $\mathcal{L}_{\rm MP}(D_{X;H})=0$ almost surely. In particular, the functional $\mathcal{F}_{X;H}$ is continuous at every point configuration avoiding $D_{X;H}.$ Arguing analogously to \cite[Lemma 2.1]{MR4682953}, we conclude that $\mathcal{F}_{X;H}$ is continuous at $\mathcal{L}_{\rm MP}$ almost surely. Therefore, the continuous mapping theorem gives the convergence in distribution 
\begin{align*}
\mathcal{F}_{X;H}(\mathcal{L}_g) \xrightarrow{\;\;d\;\;} \mathcal{F}_{X;H}(\mathcal{L}_{\rm MP})
\end{align*}
as $g \to \infty,$ and the proposition follows from (\ref{eq:psi=F}).
\end{proof}

\end{proposition}

\section{Moments of $f_{X;H}$}

In this section, we estimate the first three moments of $f_{X;H}$ with respect to $\nu_{\rm MP},$ with the bulk of the work already carried out in \cite{rudnick2026closedgeodesicsshortintervals}.

\begin{lemma} \label{lem:1}
Let $1 \leq H \leq X.$ Then
\begin{align*}
\int_{0}^{\infty} f_{X;H}(x) d\nu_{\rm MP} (x) =H+O(HX^{-1/2}).
\end{align*}

\begin{proof}
See \cite[pp. 10-11]{rudnick2026closedgeodesicsshortintervals} for the estimation, and also Remark \ref{rmk}.
\end{proof}

\end{lemma}

\begin{lemma}  \label{lem:2}
Let $1 \leq H \leq X.$ Then
\begin{align*}
\int_{0}^{\infty} f^2_{X;H}(x) d\nu_{\rm MP} (x)=
2H\log X+O \left(
\frac{H^2}{X}+\frac{H\log^2 X}{\sqrt{X}}  
\right).
\end{align*}

\begin{proof}
Following \cite[pp. 14-16]{rudnick2026closedgeodesicsshortintervals}, we have
\begin{gather} 
\int_{0}^{\infty} f^2_{X;H}(x) d\nu_{\rm MP} (x)
= \frac{1}{2}\int_{A}^{B} (2x)^2\left( e^{x/2}-e^{-x/2} \right)^2 \frac{d x}{x}  
\nonumber \\
+ 
O \left( \frac{H\log^2 X}{\sqrt{X}}  + \frac{H^2\log^2 X}{X^2} +\frac{H}{X} \right),
\label{eq:b4int}
\end{gather}
where $A=\log X$ and $B=\log(X+H)$ (see Remark \ref{rmk}). Since
\begin{align*}
\int x^2\left( e^{x/2}-e^{-x/2} \right)^2 \frac{d x}{x} =&
 \int x \left( e^{x}-2+e^{-x} \right)^2  d x \\
=& (x-1)e^x-x^2-(x+1)e^{-x}+C
\end{align*}
for any constant $C,$ it follows that
\begin{gather*}
\int_{A}^{B} x^2\left( e^{x/2}-e^{-x/2} \right)^2 \frac{d x}{x}  -H\log X  \\
 =(X+H)\delta-H-(2A\delta+\delta^2)+\frac{A+1}{X}-\frac{B+1}{X+H},
\end{gather*}
where $\delta:=\log(1+(H/X)).$ Using the fact $\delta \ll H/X$ for $1 \leq H \leq X,$ this is
\begin{align*}
\ll  \frac{H^2}{X}+\frac{H\log X}{X}.
\end{align*}
Therefore, we obtain
\begin{align*}
 \frac{1}{2}\int_{A}^{B} (2x)^2\left( e^{x/2}-e^{-x/2} \right)^2 \frac{d x}{x}
 =2H\log X+ O \left( \frac{H^2}{X}+\frac{H\log X}{X} \right),
\end{align*}
and the lemma follows from (\ref{eq:b4int}). 
\end{proof}

\end{lemma}

\begin{lemma}  \label{lem:3}
Let $1 \leq H \leq X.$ Then
\begin{align*}
\int_{0}^{\infty} f^3_{X;H}(x) d\nu_{\rm MP} (x)
\ll  H\log^2 X.
\end{align*}

\begin{proof}
By definition, we have $0 \leq f_{X;H}(x) \leq 2B=2\log(X+H)$ for any $x \geq 0.$
Applying Lemma \ref{lem:2}, we obtain
\begin{align*}
\int_{0}^{\infty} f^3_{X;H}(x) d\nu_{\rm MP} (x)
\ll&  \log X \int_{0}^{\infty} f^2_{X;H}(x) d\nu_{\rm MP} (x) \\
\ll& H\log^2 X,
\end{align*}
and therefore the lemma follows.
\end{proof}

\end{lemma}

\section{Proof of Theorem \ref{thm:main}}

We adapt the argument from \cite[Section 3]{MR4682953}. Applying Proposition \ref{prop:compare} with the function $F(x)=\exp(it(x-H)/\sqrt{2H\log X})$ for each $t \in \mathbb{R},$ we obtain
\begin{align*}
\lim_{g \to \infty} \EWp \exp \left( it \cdot \frac{\Psi_{M_g}(X;H)-H}{\sqrt{2H\log X}} \right)
=\mathbb{E} \exp \left( it \cdot \frac{\Sigma_{\rm MP}(X;H)-H}{\sqrt{2H\log X}} \right).
\end{align*}
By Campbell's formula (see \cite[Chapter 3.2]{MR1207584} for instance), we have
\begin{align*}
\mathbb{E} \exp(it \cdot\Sigma_{\rm MP}(X;H))
=
\exp \left( \int_0^{\infty} \left(e^{itf_{X;H}(x)}-1 \right) d\nu_{\rm MP}(x) \right),
\end{align*}
so that
\begin{gather} 
\log \left(\mathbb{E}\exp \left( it \cdot \frac{\Sigma_{\rm MP}(X;H)-H}{\sqrt{2H\log X}} \right) \right) \nonumber \\
=
\int_0^{\infty} \left(\exp\left({it \cdot \frac{f_{X;H}(x)}{\sqrt{2H\log X}}}\right)-1 \right) d\nu_{\rm MP}(x)  -it \cdot \frac{H}{\sqrt{2H\log X}}. \label{eq:aftercampbell}
\end{gather}
Write
\begin{align*}
\mu_{X;H}:=\int_{0}^{\infty} f_{X;H}(x) d\nu_{\rm MP} (x).
\end{align*}
Using the fact
\begin{align*}
e^{iu}=1+iu-\frac{1}{2}u^2+O(|u|^3)
\end{align*}
for any $u \in \mathbb{R},$ the expression (\ref{eq:aftercampbell}) becomes
\begin{gather*}
\frac{it(\mu_{X;H}-H)}{\sqrt{2H\log X}}
-\frac{t^2}{4H\log X} \int_0^{\infty} f_{X;H}^2(x) d\nu_{\rm MP}(x) \\
+O \left( \frac{|t|^3}{(H\log X)^{3/2}} \int_{0}^{\infty}  f_{X;H}^3(x) d\nu_{\rm MP}(x) \right).
\end{gather*}
Applying Lemma \ref{lem:1}, Lemma \ref{lem:2}, and Lemma \ref{lem:3}, this is
\begin{align*}
-\frac{t^2}{2}+O \left( |t| \sqrt{\frac{H}{X\log X}} +
|t|^2 \frac{H}{X\log X}+|t|^2 \frac{\log X}{\sqrt{X}}+
|t|^3 \sqrt{\frac{\log X}{H}}
\right).
\end{align*}
Therefore, for each $t \in \mathbb{R},$ we conclude that 
\begin{align*}
\lim_{X \to \infty}\lim_{g \to \infty} \EWp \exp \left( it \cdot \frac{\Psi_{M}(X;H)-H}{\sqrt{2H\log X}} \right)=\exp\left(-\frac{t^2}{2}\right)
\end{align*}
whenever $1 \leq H \leq X$ and $H/\log X \to \infty$ as $X \to \infty,$ and the theorem then follows by Lévy’s continuity theorem.






\section*{Acknowledgements}
The author would like to thank Mo Dick Wong for reading an early draft.

\clearpage

\printbibliography


\end{document}